\documentclass[a4paper]{amsart}

\usepackage[shortlabels]{enumitem}
\setlist[enumerate]{label={\arabic*.}}

\usepackage[dvipsnames]{xcolor}
\usepackage[backref]{hyperref}
\hypersetup{colorlinks=true,linkcolor=Maroon,citecolor=OliveGreen}
\usepackage{mathabx}
\setlist[description]{font=\normalfont\emph,leftmargin=*}
\usepackage{xfrac}
\usepackage{stmaryrd}
\usepackage{comment}
\usepackage{booktabs}
\usepackage{aliascnt}

\usepackage{amssymb}
\usepackage{amsthm}
\usepackage{mathtools}
\usepackage{amsmath,amsfonts,amsthm,mathtools,commath,amssymb}

\usepackage{cleveref}
\usepackage{microtype}

\newtheorem{theorem}{Theorem}[section]
\crefname{theorem}{theorem}{theorems}
\Crefname{theorem}{Theorem}{Theorems}

\newcommand{\newsharedtheorem}[4]{%
  \newaliascnt{#1}{theorem}%
  \newtheorem{#1}[#1]{#2}%
  \aliascntresetthe{#1}%
  \crefname{#1}{#3}{#4}%
  \Crefname{#1}{#2}{#2s}%
}

\newsharedtheorem{lemma}{Lemma}{lemma}{lemmas}
\newsharedtheorem{proposition}{Proposition}{proposition}{propositions}
\newsharedtheorem{corollary}{Corollary}{corollary}{corollaries}

\newtheorem*{theorem*}{Theorem}
\newtheorem*{example*}{Example}

\theoremstyle{definition}

\newtheorem{example}[theorem]{Example}

\newtheorem*{conjecture*}{Conjecture}

\usepackage{todonotes}

\usepackage{bbm}
\def\1{\mathbf{1}}
\def\F{\mathbf{F}}

\def\N{\mathbf{N}}

\def\Z{\mathbf{Z}}
\def\P{\mathbf{P}}
\def\E{\mathbf{E}}

\renewcommand{\hom}{\mathrm{Hom}}

\DeclareMathOperator{\Mat}{M}
\DeclareMathOperator{\End}{End}

\DeclareMathOperator{\codim}{codim}

\DeclareMathOperator{\slfrak}{\mathfrak{sl}}
\DeclareMathOperator{\glfrak}{\mathfrak{gl}}

\DeclareMathOperator{\pslfrak}{\mathfrak{psl}}
\DeclareMathOperator{\diag}{diag}
\DeclareMathOperator{\tr}{Tr}

\DeclareMathOperator{\GL}{GL}

\DeclareMathOperator{\Fr}{Fr}

\DeclareMathOperator{\mult}{mult}
\DeclareMathOperator{\ad}{ad}

\DeclareMathOperator{\Diff}{Diff}
\DeclareMathOperator{\Sum}{Sum}
\DeclareMathOperator{\Roots}{Roots}
\DeclareMathOperator{\Irr}{Irr}

\DeclareMathOperator{\id}{id}

\begin{document}
\baselineskip=13pt 

\title[Random generation of $\slfrak_n$ over finite fields]{Random generation of the special linear Lie algebra over finite fields}

\author{Urban Jezernik}
\address{Urban Jezernik, Faculty of Mathematics and Physics, University of Ljubljana, Jadranska 21, 1000 Ljubljana, Slovenia; Institute of Mathematics, Physics, and Mechanics, Jadranska 19, 1000 Ljubljana, Slovenia}
\email{urban.jezernik@fmf.uni-lj.si}

\author{Andoni Zozaya}
\address{Andoni Zozaya, Department of Statistics, Computer Science and Mathematics, Public University of Navarre, Campus of Arrosadia, 31006 Pamplona, Spain; INAMAT$^2$, Campus of Arrosadia, 31006 Pamplona, Spain}
\email{andoni.zozaya@unavarra.es}

\thanks{UJ acknowledges support of the Slovenian Research Agency (P1-0222, J1-50001, J1-70033).}

\begin{abstract}
  We prove that the special linear Lie algebra $\slfrak_n(\F_q)$ over a finite field of characteristic $p$ is generated by two random elements with high probability as $\abs{\slfrak_n(\F_q)}$ tends to infinity, provided that $(n,p) \neq (3,3), (4,2)$.
\end{abstract}

\maketitle


\section{Introduction}
\label{sec:introduction}

\subsection{Two-generation}
 
Across the simple objects in algebra, very few elements are needed to generate the whole structure. In fact, \emph{a pair typically suffices}. Over an algebraically closed field, the matrix algebra is generated by any two matrices with no common invariant subspace by Burnside's irreducibility theorem, a simple algebraic group is topologically $2$-generated \cite{guralnick1998some}, and a simple Lie algebra is $2$-generated in characteristic zero \cite{kuranishi1951everywhere} and in characteristic $p > 3$ \cite{bois2009generators}. Similar results hold over finite fields. Matrix algebras admit explicit generating pairs \cite{petrenko2007pairs}, and finite simple groups are $2$-generated \cite{steinberg1962generators,aschbacher1984some} in general. Finite simple Lie algebras are comparatively underexplored. For $\slfrak_n(\F_q)$ we showed $2$-generation in \cite{cantor2024two} with exceptions $(n,p) = (3,3), (4,2)$, where $p$ is the characteristic of $\F_q$.
 
\subsection{Random two-generation}

It turns out that often \emph{a typical pair suffices}. Over an algebraically closed field this is immediate: the generating pairs form a nonempty Zariski open set, so a generic pair generates. Over finite fields the corresponding results are deeper. For finite simple groups, two uniformly random elements generate with probability tending to $1$ as the group grows, by Kantor and Lubotzky \cite{kantor1990probability} for the classical groups and by Liebeck and Shalev \cite{liebeck1995probability} in general. The mechanism is that a random pair fails only when both elements lie in a common maximal subgroup, and the classification of finite simple groups makes it possible to control this situation. More recently, a classification-free route for classical groups was found \cite{eberhard2020random, eberhard2022babai}, relying instead on random walks and character estimates. For matrix algebras, the analogous statement is also known \cite{neumann1995cyclic, sercombe2024random}, again resting on the classification of maximal subalgebras. For Lie algebras over finite fields, less is known for structural reasons. The maximal subalgebras of even the most classical $\slfrak_n(\F_q)$ are not classified, so the argument used for groups and matrix algebras has no footing. Nor does the alternative group theoretic route via random walks have an evident Lie algebra counterpart. A different approach is required.

\subsection{Main result}
 
In this paper, we show random generation of the special linear Lie algebra by a pair of its elements.

\begin{theorem} \label{thm:main}
Assume $(n,p) \neq (3,3), (4,2)$. Two uniformly random elements of $\slfrak_n(\F_q)$ generate the Lie algebra $\slfrak_n(\F_q)$ with probability $1 - o(1)$ as $\abs{\slfrak_n(\F_q)} \to \infty$.
\end{theorem}
 
The excluded pairs $(3,3), (4,2)$ are precisely those for which $\slfrak_n(\F_q)$ fails to be $2$-generated at all \cite{cantor2024two}. The main content of the theorem lies in the regime when $n$ tends to infinity.\footnote{In fact, it follows from our proof that the probability of failure to generate decays faster than any polynomial in $n^{-1}$ in this regime.} This is because when $n$ is fixed and $q \to \infty$, the result is known to hold for all classical Lie algebras \cite[Proposition 4.5]{barbieri2025diameter} by a soft argument. In brief, when $\slfrak_n(\F_q)$ is $2$-generated, the generating pairs form a nonempty Zariski open set over $\overline{\F}_q$, and a nonempty open set cut out by equations of bounded degree depending only on $n$ is numerically dense over $\F_q$ as $q \to \infty$. This reasoning collapses when $q$ is a fixed small prime, because the degrees of the defining equations grow with $n$, and the open set might contain very few $\F_q$-points. Algebraic geometry is silent here and one must work with the structure of $\slfrak_n(\F_q)$ directly. In order to establish the above theorem, we therefore proceed by more direct means, organised around the factorisation of the characteristic polynomial of a random matrix, and particularly the additive relations among its roots.
 
\subsection{Method of proof}
 
Let $A, B \in \slfrak_n(\F_q)$ be a random pair, and let $L = \langle A, B \rangle$ be the Lie subalgebra they generate. Our argument runs as follows.
 
We begin by setting up an extraction mechanism for elementary matrices. A suitable Frobenius power of $A$ is semisimple with the same eigenvalues as $A$, and since $\slfrak_n(\F_q)$ is $p$-restricted, every iterated bracket with this power still lies in $L$. We may therefore replace $A$ by this power and assume it diagonal over $\overline{\F}_q$, so that the adjoint of $A$ scales each elementary matrix $E_{ij}$ by the eigenvalue difference $\lambda_i - \lambda_j$. Applying the adjoint repeatedly to $B$ and solving the resulting Vandermonde system isolates, for each eigenvalue difference $\delta$, the projection of $B$ onto the positions where $\lambda_i - \lambda_j = \delta$. When $\delta$ is attained by a single pair of indices, this projection is a single elementary matrix $E_{ij}$, which we thereby place in $L$ provided the entry $B_{ij}$ is nonzero. Extracting $E_{ij}$ thus rests on two independent features of the pair: the difference $\lambda_i - \lambda_j$ should be unique (a condition on $A$ alone) and the corresponding entry $B_{ij}$ should be nonzero (a condition on $B$ relative to the eigenbasis of $A$).
 
The first feature is the core of the paper. We show that, with high probability, the characteristic polynomial $\chi_A$ has an irreducible factor of degree of order $\log n$, of multiplicity one, all of whose root differences $\lambda_i - \lambda_j$ are pairwise distinct. This boils down to a counting problem. We use Reiner's formula \cite{reiner1961number} to reduce the enumeration of trace zero matrices with a given characteristic polynomial to that of the polynomials themselves. After that, we rely on estimates for smooth polynomials \cite{hildebrand1993integers,kuttner2022enumeration} to show that $\chi_A$ rarely has only small factors, and subsequently on the theory of linearized polynomials \cite{LN} to show that a largish irreducible factor rarely has a repeated difference of roots. This idea of exploiting distinct differences goes back to Steinberg \cite{steinberg1962generators} and was already crucial in our previous work \cite{cantor2024two}.
 
We then turn to the second feature, the entries of $B$ on the eigenblock attached to such a factor. The change of basis diagonalising $A$ is determined by $A$ alone, so these entries are no longer independent or uniform. Nevertheless they obey a precise law, governed by uniformly random elements of a finite extension of $\F_q$ connected by Frobenius powers. By this block law the nonzero entries occupy the cyclic diagonals indexed by a random set $D$, and as soon as $D$ generates the cyclic group, which happens with high probability, the extracted elementary matrices propagate to a full special linear block $\slfrak_m(\overline{\F}_q)$ with $m \to \infty$ as $n \to \infty$.
 
Finally, we let this block spread to the whole algebra. We show that a random pair $A,B$ acts absolutely irreducibly on the underlying module by passing to the associative algebra the pair generates and invoking the classical theory of Schur, Wedderburn and Jacobson. After that, we show that an irreducibly acting subalgebra of $\slfrak_n(\overline{\F}_q)$ containing a single special linear block must in fact be all of $\slfrak_n(\overline{\F}_q)$. This completes the proof in odd characteristic.
 
In characteristic $2$ the eigenvalue differences degenerate, since $\lambda_i - \lambda_j = \lambda_j - \lambda_i$, so no difference is ever unique and the Vandermonde step recovers only a linear combination of elementary matrices $E_{ij}$ and $E_{ji}$. We adapt the argument by working throughout with unordered sums $\lambda_i + \lambda_j$ instead, and we break the symmetry between $E_{ij}$ and $E_{ji}$ through an additional argument. The remaining parts of the proof work just as in odd characteristic.
 
\subsection{Reader's guide}
 
\Cref{sec:extraction} sets up the extraction mechanism: after replacing one matrix by a semisimple power and diagonalising it, the adjoint action together with a Vandermonde argument isolates an elementary matrix at each position carrying a unique eigenvalue difference and a nonzero entry of the second matrix. \Cref{sec:differences-of-roots} studies the differences of roots of irreducible polynomials over finite fields, and uses linearized polynomials to show that a repeated difference is rare. \Cref{sec:char-poly-of-random-matrices} combines this with estimates for smooth polynomials of prescribed trace to produce, with high probability, an irreducible factor of the characteristic polynomial of intermediate degree, of multiplicity one, and with all root differences distinct. \Cref{sec:block-generation} determines the joint law of the entries of the second matrix on such a block and shows that the elementary matrices supported on its live cyclic diagonals generate the whole block. \Cref{sec:obtaining-a-largish-block} assembles these inputs to exhibit a large special linear block inside the generated algebra. \Cref{sec:absolute-irreducibility} shows that a random pair acts absolutely irreducibly. \Cref{sec:largish-block-and-irreducibility-give-everything} contains the structural result that a block together with irreducibility forces the whole algebra, completing the proof in odd characteristic. \Cref{sec:even-characteristic} treats characteristic $2$, where differences of eigenvalues are replaced by sums and the resulting symmetry is broken by hand.

\section{Extracting elementary matrices}
\label{sec:extraction}

Throughout this section $A, B \in \slfrak_n(\F_q)$ are a pair of matrices, and we write $\langle A, B \rangle$ for the Lie subalgebra they generate. We describe the mechanism by which iterated brackets of $A$ and $B$ produce elementary matrices. It rests on making one of the two matrices semisimple and reading off the action of its adjoint in an eigenbasis. As we will see, extracting a given elementary matrix requires control of two independent features of the pair, the differences of the eigenvalues of $A$ and the vanishing of the entries of $B$.

\subsection{Making a matrix semisimple}

We first arrange for one of the matrices to be diagonalizable over $\overline{\F}_q$, at no cost to the generated Lie algebra.

\begin{lemma} \label{power_diagonalizable}
Let $A \in \Mat_n(\F_q)$. Then there is an integer $N \geq 0$ such that $A^{q^N}$ is semisimple with the same eigenvalues as $A$.
\end{lemma}
\begin{proof}
Let $\E$ be a finite extension of $\F_q$ containing all eigenvalues of $A$. Let $N$ be a multiple of $|\E : \F_q|$ such that $q^N \geq n$. Let $A = A_{\text{s}} + A_{\text{n}}$ be the Jordan-Chevalley decomposition of $A$, where $A_{\text{s}}$ and $A_{\text{n}}$ are the semisimple and nilpotent parts of $A$ that commute. We have $A_{\text{n}}^n = 0$ and $A_{\text{s}}^{\abs{\E}} = A_{\text{s}}$ since $A_{\text{s}}$ is diagonalizable over $\E$. Thus $A^{q^N} = A_{\text{s}}^{q^N} = A_{\text{s}}$.
\end{proof}

Replacing $A$ by the semisimple power $A^{q^N}$ does not take us outside the generated Lie algebra, because $\slfrak_n(\F_q)$ is $p$-restricted. We record the precise form we will use.

\begin{lemma} \label{dim_adNXY}
Let $A,B \in \slfrak_n(\F_q)$. Let $N\geq 0$ be such that $A^{q^N}$ is semisimple. Choose $P\in \GL_n(\overline{\F}_q)$ such that $X=P^{-1}A^{q^N}P$ is diagonal, and set $Y=P^{-1}BP$. Then the linear span of $\ad_X^\N Y$ is contained in $P^{-1} \langle A, B \rangle P$.
\end{lemma}
\begin{proof}
Let $A' = A^{q^N}$. Since $\slfrak_n(\F_q)$ is a $p$-restricted Lie algebra, we have
\[
\ad_{A'} B = \ad_A^{q^N} B.
\]
Hence for all $k \geq 1$ the matrix $\ad_{A'}^k B$ belongs to $\langle A, B \rangle$. Conjugating with $P$, we obtain $P^{-1} (\ad_{A'}^k B) P =  \ad_X^k Y \in P^{-1} \langle A, B \rangle P$. The claim follows.
\end{proof}

\subsection{Isolating differences of eigenvalues}

Suppose now that $A \in \Mat_n(\overline{\F}_q)$ is diagonal with eigenvalues $\Lambda = ( \lambda_1, \lambda_2, \dots, \lambda_n )$ in some order. The adjoint of $A$ scales each elementary matrix as $\ad_A(E_{ij}) = (\lambda_i - \lambda_j) E_{ij}$, so it is natural to organise the entries of the second matrix $B$ by the eigenvalue differences they sit at. We record the relevant differences in the \emph{multiset}
\[
\Diff(\Lambda) = \{ \lambda_i - \lambda_j \mid 1 \leq i, j \leq n, \ \lambda_i \neq \lambda_j \},
\]
and for $\delta \in \Diff(\Lambda)$ we set
\[
  I_\delta = \{ (i,j) \mid \lambda_i - \lambda_j = \delta \},
  \qquad
  B_\delta = \sum_{(i,j) \in I_\delta} B_{ij} E_{ij} \in \Mat_n(\overline{\F}_q).
\]
Thus $B_\delta$ is the projection of $B$ onto the elementary matrices $E_{ij}$ whose position carries eigenvalue difference $\delta$. These projections are accessible from the adjoint action alone.

\begin{lemma} \label{lie_algebra_contains_B_delta}
Let $A \in \Mat_n(\overline{\F}_q)$ be a diagonal matrix with eigenvalues $\Lambda$ in order. For any $B \in \Mat_n(\overline{\F}_q)$, the linear span of 
\[
\ad_A^{\N}(B) = \left\{ \ad_A^k(B) \mid k \in \N \right\}
\]
contains $B_\delta$ for every $\delta \in \Diff(\Lambda)$.
\end{lemma}
\begin{proof}
Let $A = \diag(\lambda_1, \lambda_2, \dots, \lambda_n)$. Then $\ad_A(E_{ij}) = (\lambda_i - \lambda_j) E_{ij}$ for any $i,j$. Thus, for all $k \geq 1$,
\[
  \ad_A^k B = \sum_{\delta} \delta^k B_\delta,
\]
where $\delta$ runs over the distinct nonzero elements of $\Diff(\Lambda)$. As $k$ varies from $1$ to $\abs{\Diff(\Lambda)}$, we obtain a linear system in the variables $B_\delta$ with coefficients forming an invertible Vandermonde matrix, since $\delta$ runs over distinct differences. The matrices $B_\delta$ can therefore be obtained as a linear combination of the $\ad_A^k(B)$ for $1 \leq k \leq \abs{\Diff(\Lambda)}$.
\end{proof}

\subsection{The extraction principle}

Combining the two ingredients yields the main principle that drives this paper. When an eigenvalue difference $\delta$ is attained by a \emph{single} pair of indices, the projection $B_\delta$ is a single elementary matrix, which can be expressed as a linear combination of $\ad_A^{\N}(B)$ as long as the corresponding entry of $B$ is nonzero.

\begin{proposition} \label{prop:extraction}
Let $A, B \in \slfrak_n(\F_q)$, let $N \geq 0$ be such that $A^{q^N}$ is semisimple, and choose $P \in \GL_n(\overline{\F}_q)$ so that $X = P^{-1} A^{q^N} P$ is diagonal with eigenvalues $\Lambda = (\lambda_1, \dots, \lambda_n)$. Write $Y = P^{-1} B P$. If $\delta = \lambda_i - \lambda_j$ has multiplicity $1$ in $\Diff(\Lambda)$ and $Y_{ij} \neq 0$, then $P^{-1} \langle A, B \rangle P$ contains the elementary matrix $E_{ij}$.
\end{proposition}
\begin{proof}
By \Cref{dim_adNXY} the span of $\ad_X^\N Y$ is contained in $P^{-1} \langle A, B \rangle P$, and by \Cref{lie_algebra_contains_B_delta} it contains $Y_\delta = \sum_{(k,l) \in I_\delta} Y_{kl} E_{kl}$. As $\delta$ has multiplicity $1$ in $\Diff(\Lambda)$, the index set $I_\delta = \{ (i,j) \}$ consists of a single pair, so $Y_\delta = Y_{ij} E_{ij}$. Since $Y_{ij} \neq 0$, dividing by it gives $E_{ij} \in P^{-1} \langle A, B \rangle P$.
\end{proof}

The two hypotheses of the preceding proposition are of an entirely different nature, and we control them separately. That $\delta$ be a unique difference is a condition on the eigenvalues of $A$ alone, while that $Y_{ij} \neq 0$ is a condition on the matrix $B$ relative to the eigenbasis of $A$. 

We first turn to the eigenvalues. The challenge is that over finite fields of bounded size, it is \emph{not} true that a uniformly random matrix in $\Mat_n(\F_q)$ is regular semisimple, i.e. has no repeated roots over $\overline{\F}_q$. In fact, in the large $n$ limit, the proportion of regular semisimple matrices in $\Mat_n(\F_q)$ is $\prod_{r = 1}^\infty (1-q^{-r}) \leq 1-q^{-1}$ (see \cite{FMP05}), which is bounded away from $1$ when $q$ is bounded. We must therefore work harder to produce eigenvalue differences with multiplicity $1$. Our strategy is to show that the characteristic polynomial of a random matrix has a \emph{largish} irreducible factor all of whose root differences are unique, and we can use the extraction principle with this factor.

\section{Differences of roots}
\label{sec:differences-of-roots}

In this section we study the differences of roots of an irreducible polynomial over $\F_q$ and show that a random irreducible polynomial of large degree rarely has a repeated difference. The estimates rest on the theory of linearized polynomials, which lets us bound how often Frobenius powers of a random field element satisfy an additive relation. Throughout this section we assume that $q$ is odd.

\subsection{Linearized polynomials}

Let $\E$ be a finite extension of $\F_q$. Let $\Fr$ be the Frobenius automorphism $x \mapsto x^q$ of $\E$. This is a linear map over $\F_q$. For any polynomial $L$ with coefficients in $\F_q$, we can thus consider $L(\Fr)$ as an endomorphism of $\E$. Such endomorphisms are called \emph{linearized polynomials}. These polynomials are well-studied in finite field theory. In particular, they have the following property.

\begin{theorem}[{\cite[determining roots of linearized polynomials, page 110]{LN}}]
Let $L$ be a polynomial over $\F_q$. Then $\dim \ker L(\Fr) \leq \deg L$.
\end{theorem}

We will use this for a set of linearized polynomials in the following probabilistic manner.

\begin{lemma} \label{prob_frob_poly_sends_alpha_to_some_set}
Let $\mathcal L$ be a finite set of nonzero polynomials over $\F_q$ and let $S \subseteq \E$.
For a uniformly random $\alpha \in \E$, we have
  \[
    \P(L(\Fr)(\alpha) \in S \text{ for some $L \in \mathcal L$}) \leq \abs{\mathcal L} \abs{S} q^{\max \deg \mathcal L} / \abs{\E}.
  \]
  \end{lemma}
  \begin{proof}
  For each $L \in \mathcal L$, the map $L(\Fr)$ is linear over $\F_q$ and has kernel of size at most $q^{\deg L}$ by the previous theorem. Hence all fibers of $L(\Fr)$ are of size at most $q^{\max \deg \mathcal L}$.   The proof follows by a union bound over all elements of $\mathcal L$ and $S$.
  \end{proof}

\subsection{Irreducible polynomials with repeated differences of roots}

Let $f$ be a monic irreducible polynomial of degree $n$ over $\F_q$. Let $\Roots(f)$ be the set of its roots in its splitting field $\E = \F_{q^n}$ and consider the \emph{multiset} of its nontrivial differences
\[
\Diff(f) = \{ \alpha - \beta \mid \alpha, \beta \in \Roots(f), \ \alpha \neq \beta \}.
\]
We say \emph{a difference is repeated} if $\Diff(f)$ has an element of multiplicity larger than $1$, {\it i.e.} $\max \mult \Diff(f) > 1$. Otherwise, we say that \emph{all differences are unique}.

\begin{example}
There are polynomials for which \emph{every} difference is repeated. For example, consider the Artin-Schreier polynomial $f(x) = x^p - x - 1$ over $\F_p$. For any root $\alpha$ of $f$, we have $\Fr^i(\alpha) = \alpha + i$ for every $i \in \F_p$, and so 
the numbers $\alpha + i$ form all the roots of $f$. Hence the differences of $f$ are just the differences $i - j$ of the elements of $\F_p$, and these are all repeated, for example $i - j = (i+1) - (j+1)$.
\end{example}

\begin{proposition} \label{prop:number_of_polynomials_with_repeated_difference}
The number of monic irreducible polynomials of degree $n$ over $\F_q$ with a repeated difference of roots is at most $n^3 q^{3n/4}$.
\end{proposition} 
\begin{proof}
Let $f$ have a root $\alpha$ and a repeated difference $\delta$. Then we have
\[
  \delta = \Fr^i(\alpha) - \Fr^j(\alpha) = \Fr^k(\alpha) - \Fr^l(\alpha)
\]
for some $0 \leq i,j,k,l < n$ with $(i,j) \neq (k,l)$. After applying a suitable power of $\Fr$, we can assume that $l=0$ and $\max \{i, j, k \} \leq 3n/4$.\footnote{After translating all four indices cyclically, we may place one of them at $0$. Among the four points on the cyclic group $\Z/n\Z$, there is a gap of length at least $n/4$. Rotate so that this gap lies at the end of the interval. Then all remaining indices lie in an interval of length at most $3n/4$.} Letting $L(x) = x^i - x^j - x^k + 1$, we thus have $\deg L \leq 3n/4$ and $\alpha \in \ker L(\Fr)$. It now follows from \Cref{prob_frob_poly_sends_alpha_to_some_set} that the number of roots $\alpha$ of $f$ is at most $q^{3n/4}$. Any $\alpha$ uniquely determines $f$. The number of possibilities for $i,j,k$ is at most $n^3$. In total there are thus at most $n^3 q^{3n/4}$ such polynomials $f$.
\end{proof}

We will require the following more general version of the above proposition.

\begin{proposition} \label{prob_some_diff_has_higher_multiplicity}
Let $h$ be a non-constant polynomial over $\F_q$. Let $f$ be a uniformly random irreducible polynomial over $\F_q$ of degree $n \geq 4$. Then
\[
  \P(\text{some $\delta \in \Diff(f)$ satisfies } \mult \delta > 1 \text{ in $\Diff(fh)$}) \leq 6 (\deg h)^2 n^3 q^{-n/4} .
\]
\end{proposition} 
\begin{proof}
An irreducible polynomial $f$ of degree $n$ corresponds to a generator of the field $\E = \F_{q^n}$. Therefore, the event under consideration is equivalent to the following: for a uniformly random generator $\alpha\in \E$, one of the three conditions holds:
\begin{itemize}
  \item $\Fr^i(\alpha) - \Fr^j(\alpha) = \Fr^k(\alpha) - \alpha$ for some $0 \leq i,j,k \leq 3n/4$;
  \item $\Fr^i(\alpha) - \Fr^j(\alpha) = \Fr^k(\alpha) - a$ for some $0 \leq i,j,k \leq 2n/3$ and $a\in \Roots(h)$;
  \item $\Fr^i(\alpha) - \alpha = a - b$ for some $0 \leq i\leq n/2$ and $a,b\in \Roots(h)$.
\end{itemize}
By \Cref{prob_frob_poly_sends_alpha_to_some_set}, if $\alpha \in \E$ is uniformly random, then by applying the union bound the probability of any of these events occurring is at most
\[ 
n^3 q^{-n/4} 
+ 
n^3 q^{-n/3} \abs{\Roots(h)} 
+ 
n q^{-n/2} \abs{\Roots(h)}^2
\leq 
3 n^3 q^{-n/4} \abs{\Roots(h)}^2.
\]
Note that the probability a uniformly random $\alpha \in \E$ fails to be a generator is bounded by
\[
\sum_{\substack{d \mid n \\ d \neq n}} 
\P (\alpha \in \F_{q^d} ) 
\leq 
\sum_{d=1}^{\lfloor n/2 \rfloor}
q^{d-n} \leq 2 q^{-n/2}.
\]
For $n \geq 4$, this probability is at most $1/2$. Hence, conditioning on $\alpha$ being a generator of $\E$, we obtain an upper bound for the probability of the three conditions as
\[
  3 n^3 q^{-n/4} \abs{\Roots(h)}^2 / (1 - 2 q^{-n/2})
  \leq 6 n^3 q^{-n/4} \abs{\Roots(h)}^2. \qedhere
\]
\end{proof}

\section{Characteristic polynomials of random matrices}
\label{sec:char-poly-of-random-matrices}

In this section, we study the factorisation of the characteristic polynomial $\chi_A$ of a random matrix $A$ in $\slfrak_n(\F_q)$. We prove that with high probability, every such polynomial has an irreducible factor of degree $\Omega(\log n)$, of multiplicity $1$, and with all differences of its roots unique. Throughout this section we assume that $q$ is odd.

\subsection{Counting characteristic polynomials}

We will count matrices with the desired properties by translating the problem, up to a negligible loss, to a problem of counting polynomials alone. The crucial input we need for this is the following result of Reiner giving an explicit formula for the number of matrices with a given characteristic polynomial.

\begin{theorem}[{\cite[Theorem 2]{reiner1961number}}] \label{thm:reiner_count}
Let $g \in \F_q[x]$ be a polynomial of degree $n$ with factorisation $g = f_1^{n_1} f_2^{n_2} \cdots f_k^{n_k}$ into powers of distinct irreducibles. Then the number of matrices in $\Mat_n(\F_q)$ with characteristic polynomial $g$ is
\[
q^{n^2 - n} \frac{F(q,n)}{\prod_{i=1}^k F(q^{\deg f_i}, n_i)},
\qquad
F(u,m) = \prod_{j=1}^m (1 - u^{-j}).
\]
\end{theorem}

We can use this formula for counting matrices in $\slfrak_n(\F_q)$ as well, since the trace zero condition depends only on the subleading coefficient of the characteristic polynomial (the trace of a polynomial is the sum of its roots). A further input that we will need is the following elementary inequality. Let $\Irr_d$ be the set of monic irreducible polynomials of degree $d$ over $\F_q$. 
Then $\abs{\Irr_d} \leq q^d/d$ (see \cite[Theorem 2.2]{rosen2013number}). We now combine these inputs to record the crude consequence that we shall need.

\begin{lemma}\label{lem:prob_char_poly_in_P}
There is an absolute constant $C>0$ such that the following holds. Let $\mathcal P$ be
any set of monic degree $n$ polynomials over $\F_q$ with trace zero. If
$A\in\slfrak_n(\F_q)$ is uniformly random, then
\[
\P(\chi_A\in\mathcal P)
\leq
C n^4 \frac{|\mathcal P|}{q^{n-1}}.
\]
\end{lemma}

\begin{proof}
Let $g = \prod_i f_i^{n_i}$ be a monic degree $n$ polynomial over $\F_q$, where the $f_i$ are distinct irreducibles. 
The number of matrices $A \in \slfrak_n(\F_q)$ with characteristic polynomial $g$ is given by \Cref{thm:reiner_count}. We can bound the numerator by $F(q,n) \leq 1$. In order to bound the denominator, we can estimate each factor $F(q^{\deg f_i}, n_i)$ by taking logarithms and using the inequality $-\log(1-x)\leq 2x$ for $0\leq x\leq 1/2$. We obtain, using a crude bound,
\[
\log\prod_{i} F(q^{\deg{f_i}},n_i)^{-1}
\leq
2\sum_{d\leq n}\abs{\Irr_d}\sum_{j\ge 1}q^{-jd}
\leq
4\sum_{d\leq n}\frac1d
\leq
4(1+\log n).
\]
Hence every trace zero monic polynomial $g$ of degree $n$ has at most $e^4 n^4 q^{n^2-n}$ matrices $A \in \slfrak_n(\F_q)$ with $\chi_A=g$. Dividing by $\abs{\slfrak_n(\F_q)}$ and summing over $g\in\mathcal P$ gives
\[
\P\bigl(\chi_A\in\mathcal P\bigr)\le e^4n^4q^{n^2-n}|\mathcal P| q^{1-n^2}
= e^4n^4 \frac{|\mathcal P|}{q^{n-1}}. \qedhere
\]
\end{proof}

\subsection{Smooth polynomials with prescribed trace}

A monic polynomial of degree $n$ over $\F_q$ is called \emph{$M$-smooth} if all of its irreducible factors have degree at most $M$. Their total number $S_{n,M}$ is controlled by classical estimates. The following upper bound will suffice for us.

\begin{lemma} \label{lem:upper_bound_smooth_polynomials}
There are absolute constants $C, c > 0$ such that the following holds. Let $M$ be a positive integer with $C \log n \leq M \leq n/2$. Then
\[
\frac{S_{n,M}}{q^n}
\leq
\exp\left(
  - c
  \frac{n}{M}
  \log\frac{n}{M}
  \right).
\]
\end{lemma}
\begin{proof}
Put $u=n/M \geq 2$. By the uniform estimate for smooth polynomials \cite[Theorem 1.10, Proposition 1.8, (1.18)]{gorodetsky2023uniform}, we have
\begin{equation} \label{eq:uniform_estimate_for_smooth_polynomials}
\frac{S_{n,M}}{q^n}
=
\rho(u) \exp
\left(
  O\left(\frac{n\log n}{M^2}\right)
\right)
\end{equation}
in the range $M \geq \log(n\log^2 n)/\log q$, where $\rho$ is the Dickman function. As $q\geq 2$, this range contains $M\geq C\log n$ once $C$ is large enough. Let us now bound the two factors. For the Dickman function we use the standard estimate \cite[(1.7)]{hildebrand1993integers}
\[
\rho(u)\le\exp(-c u\log u)\qquad(u\ge 2)
\]
for an absolute constant $c>0$. For the error term, the hypothesis $M\ge C\log n$
gives
\[
\frac{n\log n}{M^2}
=u \frac{\log n}{M}
\leq \frac{u}{C}.
\]
Absorb the implied constant from \cref{eq:uniform_estimate_for_smooth_polynomials} into $C$, so that the error term in the exponential is at most $u/C$. Taking $C\ge 2/(c\log 2)$, we thus have $u/C \leq (c/2) u \log u$, and so
\[
\frac{S_{n,M}}{q^n}
\leq
\exp\left( -\frac{c}{2} u\log u \right).
\]
The proof is complete.
\end{proof}

A subtlety is that we will need the count with the trace coefficient prescribed to be zero. For this we use the exact formulas of Kuttner and Wang for smooth polynomials with a prescribed trace coefficient, which crucially avoid the loss of a factor of $q$ that a naive
averaging argument relying on the preceding lemma would incur. 

\begin{theorem}[{\cite[Corollary 17]{kuttner2022enumeration}}] \label{thm: KW22}
Let $S_{n,M}(0)$ denote the number of monic $M$-smooth polynomials over $\F_q$ of characteristic $p$ of degree $n$ with trace zero. For $i \geq 1$ let
\[
a_i=\frac{1}{iq}\sum_{\substack{k\mid i\\ p\nmid k}}\mu(k)q^{i/k},
\qquad 
b_i=\frac{1}{i}\sum_{\substack{k\mid i\\ p\mid k}}\mu(k)q^{i/k},
\]
where $\mu$ is the M\"{o}bius function, and for $k \geq 1$ let
\[
A_k(a_i,b_i) = [y^k](1-(-y)^p)^{a_iq/p}(1+y)^{b_i},
\]
where $[y^n]G(y)$ is the coefficient of $y^n$ in a formal power series $G(y)$. Then
\[
S_{n,M}(0)
=
\frac{1}{q}S_{n,M}
+
\frac{q-1}{q}J_{n,M},
\qquad
J_{n,M} = [z^n]\prod_{i=M+1}^{n}
\sum_{k\ge 0}A_k(a_i,b_i)(-1)^kz^{ik}.
\]
\end{theorem}

Using this formula, we can now show that $M$-smooth polynomials of trace zero with $M \gg \log n$ are rare among all monic trace zero polynomials.

\begin{lemma}\label{lem:smooth-trace-zero}
There are absolute constants $C, c>0$ such that the following holds. Let $M$ be a positive integer with $C\log n\le M\le n/2$. Then
\[
\frac{S_{n,M}(0)}{q^{n-1}}
\le
\exp \left(- c \frac nM\log\frac nM\right)
+
q^{1-n(1-1/p)}
\exp\left(C \frac nM\log n\right).
\]
\end{lemma}
\begin{proof}
Using the explicit formula of \Cref{thm: KW22}, it suffices to control the
$S_{n,M}$ and $J_{n,M}$ terms. The first is immediate from
\Cref{lem:upper_bound_smooth_polynomials}. For the second, expand the two factors in
the definition of $A_k(a_i,b_i)$ as binomial series, so that
\[
\abs{A_k(a_i,b_i)}
\leq
\sum_{j=0}^{\lfloor k/p\rfloor}
\abs{\binom{a_iq/p}{j}}
\abs{\binom{b_i}{k-pj}}.
\]
We bound the binomial symbols using $|\binom{X}{r}| \leq (\abs X+r)^r$ for $r \geq 0$. Note that the explicit formulas for $a_i,b_i$ admit the crude bounds
\[
\abs{a_i} q / p \leq 2q^i / i,
\qquad
\abs{b_i} \leq q^{i/p}.
\]
The only indices $i,k$ that contribute to $J_{n,M}$ are those for which $M<i\le n$ and $0\le k\le n/i$. Since $i>M\ge C\log n$, we have $q^{i} \geq n$ once $C\geq 2$. Hence $k \leq n/i \leq q^i/i$, and so for any $0 \leq j \leq \lfloor k/p\rfloor$, we have
\[
\abs{\binom{a_iq/p}{j}}
\leq
(Cq^i / i)^{j},
\qquad
\abs{\binom{b_i}{k-pj}}
\leq
(Cnq^{i/p})^{k-pj}
\]
after increasing $C$. Hence
\[
\abs{A_k(a_i,b_i)}
\leq
\sum_{j=0}^{\lfloor k/p\rfloor}
(C q^i / i )^{j}
(Cnq^{i/p})^{k-pj}
\leq
(2Cn)^k q^{ik/p}.
\]
A monomial contributing to $J_{n,M}$ is determined by integers $k_i$ with $0 \leq k_i \leq n/i$, $M<i\le n$, and $\sum_{i>M}ik_i=n$. For such a tuple the preceding bound gives
\[
\prod_{i=M+1}^n\abs{A_{k_i}(a_i,b_i)}
\leq
(2Cn)^{\sum_{i>M}k_i} 
q^{\sum_{i>M}ik_i/p}
\leq
q^{n/p}
(2Cn)^{n/M}.
\]
After increasing $C$, the latter is $q^{n/p} \exp(C (n/M) \log n)$. It remains to count the contributing tuples $(k_i)_{i>M}$. Their number is at most the number of
partitions of $n$ into parts greater than $M$. A crude bound suffices here. Such a tuple has at most $n/M$ nonzero parts, and there are at most $n$ choices for each part, so the number of tuples is at most $n^{n/M}$. Combining this with the preceding estimate yields
\[
\abs{J_{n,M}}
\leq
q^{n/p} 
\exp\left(C \frac{n}{M} \log n\right)
\]
after increasing $C$. This completes the proof.
\end{proof}

\begin{proposition}\label{prop:medium-factor-exists}
There are absolute constants $C, c, n_0 >0$ such that the following holds. Let $n \geq n_0$ and let $A\in\slfrak_n(\F_q)$ be uniformly random. Then $\chi_A$ has an irreducible factor of degree $> \lceil C\log n\rceil$ with probability $1-\exp(-cn)$.
\end{proposition}
\begin{proof}
Let $c_0$ and $C_0$ be the constants of \Cref{lem:smooth-trace-zero}. We may assume $c_0 \leq 1/3$. Set $C=\max\{C_0, 4C_0/c_0\}$, $M=\lceil C\log n\rceil$. Suppose throughout that $n$ is large enough so that $M\le 2C\log n$, and $\log(n/M) \geq (\log n)/2$. By \Cref{lem:prob_char_poly_in_P} and \Cref{lem:smooth-trace-zero}, the probability that $\chi_A$ is $M$-smooth is at most
\[
C n^4\left(
\exp\left(-c_0\frac{n}{M}\log\frac{n}{M}\right)
+
q^{1-n(1-1/p)}\exp\left(C_0\frac{n}{M}\log n\right)
\right).
\]
after possibly increasing $C$. Since $M\geq C\log n$ we have $(n/M) \log n\leq n/C$, so
\[
c_0\frac{n}{M}\log\frac{n}{M} 
\geq 
\frac{c_0 n}{2C\log n} \frac{\log n}{2}
=\frac{c_0}{4C}n,
\qquad
C_0\frac{n}{M}\log n 
\leq 
\frac{C_0}{C}n 
\leq 
\frac{c_0}{4}n.
\]
As $q\geq 2$ we also have $q^{1-n(1-1/p)}\leq 2^{1-n/2}\leq\exp(-n/3)$, and since
$c_0\leq 1$ the second term in the bracket is at most $\exp(-n/3+n/4)\leq \exp(-n/12)$.
Both terms are therefore at most $\exp(-2cn)$ for $c = c_0/(8C)$, and so
\[
\P\bigl(\chi_A\text{ is }M\text{-smooth}\bigr)
\leq 2C n^4 \exp(-2cn)
\ll \exp(-cn) \]
for large enough $n$. The proof is complete.
\end{proof}



\subsection{Excluding repeated differences}

Let $\Irr_d$ be the set of monic irreducible polynomials of degree $d$ over $\F_q$, and for $t\in\F_q$ let $\Irr_{d,t} \subseteq \Irr_d$ consist of polynomials with trace $t$. Up to an error term, the trace is equidistributed in $\Irr_d$. 

\begin{lemma} \label{lem:trace-equidist}
We have
\[
\abs{
\frac{\abs{\Irr_d}}{q^d/d} - 1
}
\leq d q^{-d/2},
\qquad
\abs{
\frac{\abs{\Irr_{d,t}}}{q^{d-1}/d} - 1
}
\leq 2 d q^{-d/2 + 1}.
\]
\end{lemma}
\begin{proof}
The first inequality is \cite[Theorem 2.2]{rosen2013number}. For the second one, it follows from \cite{yucas2006irreducible} that for any $t \neq 0$ we have
\[
  \abs{\Irr_{d,t}} = \frac{1}{qd} \sum_{\substack{m \mid d \\ (m,p) = 1}} \mu(m) q^{d/m},
\]
hence
\[ 
\abs{
\abs{\Irr_{d,t}} - \frac{q^{d-1}}{d}
}
\leq
\frac{1}{qd}
\sum_{\substack{m \mid d \\ m > 1}} q^{d/m}
\leq
\frac{1}{qd} d  q^{d/2} = q^{d/2 - 1}.
\]
Dividing by $q^{d-1}/d$ gives the (slightly stronger) second inequality for $t \neq 0$. Using this, we can extract the inequality for $t = 0$ by
\[
\abs{
\frac{\abs{\Irr_{d,0}}}{q^{d-1}/d} - 1
}
=
\abs{
q \left( \frac{\abs{\Irr_d}}{q^{d}/d} - 1 \right)
-
\sum_{t\neq 0} 
\left( 
\frac{\abs{\Irr_{d,t}}}{q^{d-1}/d} 
- 1
\right)
}
\leq
2 d q^{-d/2 + 1}.
\qedhere
\]
\end{proof}

\begin{proposition}\label{prop:no-bad}
For every $\gamma>0$ there are constants $C,n_0>0$ such that the following holds. Let $n\ge n_0$ and let $A\in\slfrak_n(\F_q)$ be uniformly random. The probability that $\chi_A$ has an irreducible factor $f$ with $\deg f > \lceil C\log n\rceil$ and with some element in $\Diff(f)$ having multiplicity $> 1$ in $\Diff(\chi_A)$ is at most $n^{-\gamma}$.
\end{proposition}

\begin{proof}
Set $M=\lceil C\log n\rceil$. Fix $d>M$ and let $g = f h$ with $f$ irreducible of degree $d$ and $\deg h= n-d$. The polynomial $g$ is of trace zero if and only if $\tr f=-\tr h$. For fixed $h$, let $\mathcal D_d(h) \subseteq \Irr_d$ be the set of $f$ for which some element in $\Diff(f)$ has multiplicity $>1$ in $\Diff(fh) = \Diff(g)$. By
\Cref{prob_some_diff_has_higher_multiplicity}, for uniformly random $f\in\Irr_d$, we have
\[
\P (f\in \mathcal D_d(h) )\le 6(\deg h)^2 d^3 q^{-d/4}\le 6n^2 d^3 q^{-d/4}.
\]
Since $h$ is fixed, the trace of $f$ is prescribed to be $- \tr h$. By \Cref{lem:trace-equidist}, we have for any $t \in \F_q$ that
\[
\frac{\abs{\Irr_{d,t}}}{\abs{\Irr_d}}
\geq
\frac{q^{d-1}/d - 2 q^{d/2}}{q^d/d + q^{d/2}}
=
\frac{1}{q} \frac{1 - 2 d q^{1-d/2}}{1 + d q^{-d/2}}
\geq
\frac{1}{2q}
\]
for $n$ and therefore also $d$ large enough. Hence
\[
\frac{\abs{\mathcal D_d(h)\cap\Irr_{d,t}}}{\abs{\Irr_{d,t}}}
\le\frac{\P_{f\in\Irr_d}(f\in \mathcal D_d(h))}{\abs{\Irr_{d,t}}/\abs{\Irr_d}}
\ll n^2 d^3 q^{-d/4+1}.
\]
Taking $t=-\tr h$ and using $\abs{\Irr_{d,t}}\ll q^{d-1}/d$, it follows that
\[
\abs{\mathcal D_d(h) \cap \Irr_{d,-\tr h}}
\ll n^2 d^2 q^{d-d/4}.
\]
Summing over the $q^{n-d}$ choices of $h$, we conclude that the number of trace zero degree $n$ polynomials $g$ with a distinguished irreducible factor $f$ of degree $d$ with a high multiplicity difference in $g = fh$ is $\ll n^2 d^2 q^{n-d/4}$. It now follows from \Cref{lem:prob_char_poly_in_P} that $\chi_A$ is a polynomial of this sort for some $d > M$ with probability
\[
\ll 
\sum_{d > M} n^6 d^2 q^{1 - d/4}
\ll
n^6 \sum_{d>M} d^2 2^{-d/4}
\ll
n^6 M^2 2^{-M/4}
\sum_{\ell \geq 1} \ell^2 2^{-\ell/4}
\ll
n^6 M^2 2^{-M/4}.
\]
Since $M=\lceil C\log n\rceil$, this is $\ll n^{6-(C\log 2)/4}(\log n)^2$. Choosing $C$ large enough makes this at most $n^{-\gamma}$ for all large enough $n$.
\end{proof}

\begin{theorem}\label{thm:good-medium-factor}
For every $\gamma>0$ there are constants $C,n_0>0$ such that the following holds. Let
$n\ge n_0$, and let $A\in\slfrak_n(\F_q)$ be uniformly random. Then with probability at least $1-n^{-\gamma}$ the characteristic polynomial $\chi_A$ has an irreducible factor $f$ with $\deg f > \lceil C\log n\rceil$ and every element of $\Diff(f)$ has multiplicity $1$ in $\Diff(\chi_A)$.
\end{theorem}
\begin{proof}
Immediate from \Cref{prop:medium-factor-exists} and \Cref{prop:no-bad}.
\end{proof}

\section{Generation across the block}
\label{sec:block-generation}

On an eigenblock all of whose eigenvalue differences are unique, \Cref{prop:extraction} produces the elementary matrix $E_{ij}$ exactly at those positions where the entry $Y_{ij}$ of the conjugated second matrix is nonzero. In this section we determine the joint law of these entries on a single irreducible block, and show that the produced elementary matrices generate the whole eigenblock with high probability.

\subsection{A block law for eigenblocks}

Let $A \in \slfrak_n(\F_q)$ be semisimple, and suppose its characteristic polynomial $\chi_A$ has an irreducible factor $f$ of degree $m \leq n$ with multiplicity $1$. Then
\[
  \F_q^n = U \oplus W, 
  \qquad 
  U = \ker f(A), 
  \quad 
  W = \ker(\chi_A/f)(A)
\]
is a decomposition into $A$-invariant subspaces with $\dim U = m$. Conjugate $A$ to diagonal form respecting this decomposition by some $P \in \GL_n(\overline{\F}_q)$. For any other matrix $B \in \slfrak_n(\F_q)$, the entries of $Y = P^{-1}BP$ lying in the resulting $m \times m$ upper left eigenblock depend on $B$ only through the compression
\[
  \beta = \pi B|_U \in \End_{\F_q}(U),
\]
where $\pi \colon \F_q^n \to U$ is the projection along $W$. As $B$ ranges uniformly over $\slfrak_n(\F_q)$, the compression $\beta$ ranges uniformly over $\End_{\F_q}(U)$ if $m < n$, since the entries within $U$ are unconstrained, and the trace can be set freely using $W$ since $m < n$. If $m = n$, the matrix $\beta = B$ simply ranges uniformly over $\slfrak(U)$. The study of the eigenblock therefore reduces to the following situation: $U$ is an $m$-dimensional vector space over $\F_q$ on which $\alpha = A |_U$ acts with irreducible characteristic polynomial $f$, and $\beta \in \End_{\F_q}(U)$ (or $\slfrak(U)$) is uniformly random. The next lemma describes the matrix of $\beta$ in a carefully chosen eigenbasis of $\alpha$. We identify $U$ with an extension of $\F_q$ on which $\alpha$ acts as multiplication by some root $a$ of $f$.

\begin{lemma}\label{lem:block-law}
Let $\ell = q^m$ and view $\F_\ell$ as a vector space over $\F_q$. Let $\alpha \in \End_{\F_q}(\F_\ell)$ be multiplication by a field generator $a \in \F_\ell$, with eigenvalues $\lambda_i = \Fr^i(a)$ for $i \in \Z/m\Z$. Let $\beta \in \End_{\F_q}(\F_\ell)$ \textup{(}or $\slfrak(\F_\ell)$\textup{)} be uniformly random, and let $\beta_{ij} \in \F_\ell$ be the entries of $\beta$ in an eigenbasis $(v_i)$ of $\alpha$, normalised by $v_i = \Fr^i(v_0)$. Then there are independent uniformly random variables $Z_1, \dots, Z_{m-1} \in \F_\ell$ with
\[
  \beta_{ij} = \Fr^i(Z_{j-i})
  \qquad (i,j \in \Z/m\Z, \ i \neq j).
\]
\end{lemma}
\begin{proof}
Every $\F_q$-linear endomorphism of $\F_\ell$ is uniquely a linearized polynomial
\[
  \beta = \sum_{d=0}^{m-1} Z_d \Fr^d \qquad (Z_d \in \F_\ell),
\]
and the correspondence $\beta \mapsto (Z_0, \dots, Z_{m-1})$ is an isomorphism between the vector spaces $\End_{\F_q}(\F_\ell)$ and $\F_\ell^m$. If $\beta$ is uniformly random in $\End_{\F_q}(\F_\ell)$, then all $Z_d$ are independent and uniformly distributed in $\F_\ell$. If instead $\beta$ is uniformly random in $\slfrak(\F_\ell)$, then the trace zero condition imposes only the restriction $\tr_{\F_\ell/\F_q}(Z_0)=0$. Hence in either case the variables $Z_1, \dots, Z_{m-1}$ are independent and uniformly distributed in $\F_\ell$. Now diagonalise $\alpha$ by extending scalars to $\F_\ell$. The underlying vector space becomes
\[
  \F_\ell \otimes_{\F_q} \F_\ell \cong \prod_{i \in \Z/m\Z} \F_\ell,
  \qquad 
  x \otimes y \mapsto \bigl( \Fr^i(x) y \bigr)_i.
\]
Multiplication by $a$ acts on the $i$-th factor as multiplication by $\lambda_i$, so this factor is precisely the eigenline spanned by $v_i = \Fr^i(v_0)$. Similarly, multiplication by $Z \in \F_\ell$ acts on the $i$-th factor by $\Fr^i(Z)$, and $\Fr$ cyclically shifts factors by $(w_i)_i \mapsto (w_{i+1})_i$ for $w$ in the product. Hence
\[
  ( \beta w )_i = 
  \sum_{d=0}^{m-1} \Fr^i(Z_d) w_{i+d},
\]
and evaluating at $w = v_j$ gives $\beta_{ij} = \Fr^i(Z_{j-i})$, as claimed.
\end{proof}

Since $\Fr$ is a bijection, the entry $\beta_{ij}$ is nonzero if and only if $Z_{j-i} \neq 0$. The nonzero entries of $\beta$ therefore fill out the cyclic diagonals indexed by
\[
  D = \{ d \in \Z/m\Z \mid Z_d \neq 0, \ d \neq 0 \}.
\]

\subsection{Propagation across the block}

The generation of the eigenblock is governed by a single condition on the nonzero diagonal indices $D$.

\begin{lemma} \label{lem:propagate}
Let $\F$ be a field and let $D \subseteq \Z/m\Z \setminus \{0\}$ generate the group $\Z/m\Z$. Then the matrices $\{ E_{ij} \mid i \neq j, \ j - i \in D \}$ generate the Lie algebra $\slfrak_m(\F)$.
\end{lemma}
\begin{proof}
Let $j-i, l - j \in D$. Note that $[E_{ij}, E_{jl}] = E_{il}$ for $i \neq l$, which raises the column index of $E_{ij}$ by $l - j$ to $E_{il}$. Every element of $\Z/m\Z$ is a sum of elements of $D$. Fix $r \in \Z/m\Z$ with $r \neq 0$ and choose a representation $r = d_1 + \dots + d_k$ with all $d_s \in D$ of minimal length $k$. Its partial sums $P_t = d_1 + \dots + d_t$ for $t < k$ are all nonzero by minimality. Hence, for every $i \in \Z/m\Z$, we can form the iterated bracket
\[
  [ \cdots [[ E_{i, i + P_1}, E_{i + P_1, i + P_2} ], E_{i + P_2, i + P_3} ], \cdots ] = E_{i, i + r},
\]
where each matrix $E_{i + P_t, i + P_{t+1}}$ has column minus row index equal to $d_{t+1} \in D$ and so lies in our supply. Thus $E_{i, i + r}$ lies in the generated Lie algebra for every $i$ and every $r \neq 0$, yielding all off-diagonal elementary matrices. Finally $[E_{ij}, E_{ji}] = E_{ii} - E_{jj}$ for $i \neq j$ supplies the diagonal, and so we obtain all of $\slfrak_m(\F)$.
\end{proof}

The hypothesis that $D$ generates $\Z/m\Z$ is also necessary. If $D$ lies in a proper subgroup $H \leq \Z/m\Z$, then $[E_{ij}, E_{jl}] = E_{il}$ keeps the column minus row index in $H$, so the generated algebra meets no cyclic diagonal outside $H$ and is proper.

\subsection{The block generates}

Combining the propagation lemma with the law of \Cref{lem:block-law} gives the main result of this section.

\begin{proposition} \label{prop:block-generates}
Let $m \geq 2$, let $\ell = q^m$, and let $Z_1, \dots, Z_{m-1} \in \F_\ell$ be independent and uniform. Set $D = \{ d \in \Z/m\Z \mid Z_d \neq 0, \ d \neq 0 \}$. Then the elementary matrices $\{ E_{ij} \mid i \neq j, \ j - i \in D \}$ generate $\slfrak_m(\overline{\F}_q)$ with probability at least $1 - m \ell^{-m/2}$.
\end{proposition}
\begin{proof}
By \Cref{lem:propagate} applied over $\overline{\F}_q$, the stated elementary matrices generate $\slfrak_m(\overline{\F}_q)$ whenever $D$ generates $\Z/m\Z$. It therefore suffices to bound the probability that $D$ lies in a maximal subgroup, i.e. in a subgroup $r\Z/m\Z$ of multiples of some prime $r \mid m$. The inclusion $D \subseteq r\Z/m\Z$ forces $Z_d = 0$ for each of the $m - m/r \geq m/2$ indices $d$ with $r \nmid d$, and these are independent events, each occurring with probability $\ell^{-1}$. Hence
\[
  \P(D \text{ does not generate } \Z/m\Z)
  \leq \sum_{\substack{r \mid m \\ r \text{ prime}}} \ell^{-(m - m/r)}
  \leq m \ell^{-m/2}. \qedhere
\]
\end{proof}

\section{Obtaining a largish block}
\label{sec:obtaining-a-largish-block}

We now assemble the preceding sections. The extraction principle of \Cref{sec:extraction} turns the two features studied in \Cref{sec:char-poly-of-random-matrices} and \Cref{sec:block-generation} into a largish special linear block inside the generated algebra.

\begin{theorem} \label{largish_block}
For every $\gamma > 0$ there are constants $C, n_0 > 0$ such that the following holds. Let $q$ be odd, let $n \geq n_0$, and let $A, B \in \slfrak_n(\F_q)$ be uniformly random. Then, with probability at least $1 - n^{-\gamma}$, there are $P \in \GL_n(\overline{\F}_q)$ and $m > \lceil C \log n \rceil$ such that the Lie algebra generated by $P^{-1}AP, P^{-1}BP$ contains the upper left block $\slfrak_m(\overline{\F}_q)$.
\end{theorem}
\begin{proof}
By \Cref{power_diagonalizable}, let $N \geq 0$ be such that $A^{q^N}$ is semisimple with the same eigenvalues as $A$, say $\Lambda = (\lambda_1, \dots, \lambda_n)$.  Choose $P_0 \in \GL_n(\overline{\F}_q)$ so that $X = P_0^{-1} A^{q^N} P_0$ is diagonal, and set $Y = P_0^{-1} B P_0$. By \Cref{thm:good-medium-factor}, with probability at least $1 - n^{-2\gamma}$ the polynomial $\chi_A$ has an irreducible factor $f$ of degree $m = \deg f > \lceil C \log n \rceil$ such that every element of $\Diff(f)$ has multiplicity $1$ in $\Diff(\chi_A)$. As $m \geq 2$, the factor $f$ has multiplicity $1$ in $\chi_A$, so its roots are $m$ simple eigenvalues of $A$. After permuting the basis we place these at the first $m$ indices and order them as $\lambda_i = \Fr^i(a)$ for a root $a$ of $f$ and $i \in \Z/m\Z$, the eigenbasis being the normalised one of \Cref{lem:block-law}.

For distinct $i, j$ in the block, the difference $\delta = \lambda_i - \lambda_j$ lies in $\Diff(f)$ and so has multiplicity $1$ in $\Diff(\chi_A)$. As $\Lambda$ lists the roots of $\chi_A$ with multiplicity, $(i,j)$ is the unique pair of indices with $\lambda_i - \lambda_j = \delta$. By \Cref{prop:extraction}, the conjugate $P_0^{-1} \langle A, B \rangle P_0$ therefore contains $E_{ij}$ whenever $Y_{ij} \neq 0$. By \Cref{lem:block-law} the block entries satisfy $Y_{ij} = \Fr^i(Z_{j-i})$ for independent uniform $Z_1, \dots, Z_{m-1} \in \F_\ell$, $\ell = q^m$, so $Y_{ij} \neq 0$ exactly when $j - i \in D = \{ d \neq 0 \mid Z_d \neq 0 \}$. Hence $P_0^{-1} \langle A, B \rangle P_0$ contains every $E_{ij}$ with $j - i \in D$, and by \Cref{prop:block-generates} these generate $\slfrak_m(\overline{\F}_q)$ with probability at least $1 - m \ell^{-m/2}$.

We can absorb the permutation into a single $P \in \GL_n(\overline{\F}_q)$, so that the Lie algebra $P^{-1} \langle A, B \rangle P$ contains the upper left block $\slfrak_m(\overline{\F}_q)$. The total probability of failure is
\[
n^{-2\gamma} + m \ell^{-m/2} \leq 
n^{- 2\gamma} + C n^{-(C^2 \log q /2) \log n} \log n \leq
n^{-\gamma}
\]
for large enough $n$.
\end{proof}

\section{Absolute irreducibility}
\label{sec:absolute-irreducibility}

Let $L$ be a Lie algebra with a representation on a vector space $V$ over $\F$. Say $L$ acts \emph{irreducibly} on $V$ if there are no nontrivial $L$-invariant subspaces of $V$. Say $L$ acts \emph{absolutely irreducibly} on $V$ if $\overline{\F} \otimes L$ acts irreducibly on $\overline{\F} \otimes V$. In this section, we prove that a random pair in $\slfrak_n(\F_q)$ generates, with high probability, a Lie subalgebra that acts absolutely irreducibly on $V = \F_q^n$. Together with the largish block produced in the previous section, this is what will force the generated subalgebra to be all of $\slfrak_n(\F_q)$.

Let $A, B \in \slfrak_n(\F_q)$ be uniformly random, and let $L = \langle A, B \rangle$ be the Lie subalgebra they generate. Note that a subspace of $V$, or of $\E \otimes V$ for any extension $\E/\F_q$, is stable under $A$ and $B$ if and only if it is stable under the unital associative algebra $\mathcal A = \F_q \langle A, B \rangle \leq \Mat_n(\F_q)$ generated by $A$ and $B$. Hence $L$ and $\mathcal A$ have exactly the same invariant subspaces over every extension of $\F_q$, and consequently $L$ acts (absolutely) irreducibly on $V$ if and only if $\mathcal A$ acts (absolutely) irreducibly on $V$. This allows us to bring in the tools of Schur, Wedderburn, and Jacobson, which are statements about the associative algebra $\mathcal A$. 

If $\mathcal A$ fails to act absolutely irreducibly, then exactly one of the following occurs:
\begin{enumerate}
  \item $\mathcal A$ acts reducibly already over $\F_q$,
  \item $\mathcal A$ acts irreducibly over $\F_q$, but not absolutely irreducibly.
\end{enumerate}
We bound the probability of each case in turn.

\subsection{The reducible case}

We first record the basic counting estimate that controls reducibility over $\F_q$.

\begin{lemma} \label{lem:rational_invariant_subspace}
Let $V$ be a vector space over $\F_q$ with a proper nontrivial subspace $U$. Let $A \in \slfrak_n(\F_q)$ be uniformly random. Then
\[
  \P(AU \leq U) \leq q^{-\dim U \cdot \codim U}.
\]
\end{lemma}
\begin{proof}
Choose a basis $e_1, \dots, e_n$ of $V$ with $U = \langle e_1, \dots, e_d \rangle$. In this basis, the condition $AU \leq U$ says precisely that the lower left block of $A$ vanishes. These are $d(n-d)$ of the off-diagonal coordinate functionals on $\Mat_n(\F_q)$ that remain linearly independent after restriction to $\slfrak_n(\F_q)$. Hence the subspace they cut out has codimension exactly $d(n-d)$ in $\slfrak_n(\F_q)$.
\end{proof}

\begin{proposition} \label{prop:rational_reducible_rare}
Let $A, B \in \slfrak_n(\F_q)$ be uniformly random. The probability that $A$ and $B$ have a common proper nontrivial invariant subspace over $\F_q$ is at most $4 n q^{1-n}$.
\end{proposition}
\begin{proof}
Fix a subspace $U \leq V = \F_q^n$ of dimension $d$ with $d \neq 0,n$. Since $A$ and $B$ are independent, the preceding lemma gives
\[
  \P(AU \leq U, \ BU \leq U) \leq q^{-2d(n-d)}.
\]
There are 
\[\binom{n}{d}_q
  =
 q^{d(n-d)}
  \prod_{i=0}^{d-1}
  \frac{1-q^{i-n}}{1-q^{i-d}} 
  \leq 4 q^{d(n-d)} \]
$d$-dimensional subspaces of $\F_q^n$.  Summing over all subspaces and using $d(n-d) \geq n-1$, the probability of a common invariant subspace is at most
\[
  \sum_{d=1}^{n-1} \binom{n}{d}_q q^{-2d(n-d)}
  \leq 4 \sum_{d=1}^{n-1} q^{-d(n-d)}
  \leq 4 n q^{-(n-1)}. \qedhere
\]
\end{proof}

\subsection{The irreducible but not absolutely irreducible case}

We now turn to the second case, where the structure of the endomorphism ring of an irreducible module produces a much smaller probability.

\begin{proposition} \label{prop:irred_not_abs_rare}
Let $A, B \in \slfrak_n(\F_q)$ be uniformly random. The probability that $\mathcal A = \F_q \langle A, B \rangle$ acts irreducibly on $\F_q^n$ but not absolutely irreducibly is at most $4 n q^{2 - n^2/2}$.
\end{proposition}
\begin{proof}
Suppose that $V = \F_q^n$ is irreducible as an $\mathcal A$-module. By Schur's lemma, the centralizer
\[
  D = \End_{\mathcal A}(V) = C_{\Mat_n(\F_q)}(A, B)
\]
is a division ring. Since $D$ is finite, it follows from Wedderburn's little theorem that it is in fact a finite field, thus isomorphic to $\F_{\ell}$ with $\ell = q^k$ for some $k \geq 1$. This field as well as its embedding into $\End_{\F_q}(V)$ depends on the random matrices $A,B$. Note that $D$ acts faithfully and $\F_q$-linearly on $V$, so we may view $V$ as a vector space over $D$. In particular, $k$ must divide $n$, so let us write $m = n/k$. We inspect all possibilities for $D$ and its embedding into $\End_{\F_q}(V)$, and then do a union bound for the event that $A, B \in C_{\Mat_n(\F_q)}(D)$.

If $k = 1$, then $D \cong \F_q$, and the Jacobson density theorem implies that $\mathcal A = \End_D(V) = \Mat_n(\F_q)$. After extending scalars, the algebra $\overline{\F}_q \otimes \mathcal A = \Mat_n(\overline{\F}_q)$ clearly acts irreducibly on $\overline{\F}_q \otimes V$. Therefore in the irreducible but not absolutely irreducible case, we must have $k \geq 2$.

Let us count the number of possible embeddings of $\F_{\ell}$ into $\Mat_n(\F_q)$. For a fixed $k$ that divides $n$, all such embeddings are conjugate under $\GL_n(\F_q)$ by the Skolem--Noether theorem. The stabilizer under conjugation of a fixed copy $\F_{\ell}$ contains
\[
  C_{\GL_n(\F_q)}(\F_{\ell}) = C_{\Mat_n(\F_q)}(\F_{\ell})^\times \cong \GL_m(\F_{\ell}).
\]
The number of embedded copies of $\F_{\ell}$ is thus
\[
  \abs{\GL_n(\F_q) : \operatorname{Stab}(\F_{\ell})}
  \leq \frac{\abs{\GL_n(\F_q)}}{\abs{\GL_m(\F_{\ell})}}
  \leq \frac{q^{n^2}}{q^{n^2/k} / 4}
  = 4 q^{n^2(1 - 1/k)}.
\]

Now, for any fixed embedded copy of $\F_{\ell}$ into $\Mat_n(\F_q)$, its centralizer is precisely the $\F_{q}$-endomorphisms of $\F_q^n$ that are in fact $\F_{\ell}$-linear, and can thus be viewed as $\F_{\ell}$-linear endomorphisms of $(\F_{\ell})^m$. Hence
\[
  C_{\Mat_n(\F_q)}(\F_{\ell}) \cong \Mat_m(\F_{\ell}),
  \qquad
  \dim_{\F_q} C_{\Mat_n(\F_q)}(\F_{\ell}) = k m^2 = n^2 / k.
\]
For uniformly random $A,B$, we obtain the crude bound
\[
  \P\left( A, B \in C_{\Mat_n(\F_q)}(\F_{\ell}) \right)
  \leq
  \left( 
    \frac{|C_{\Mat_n(\F_q)}(\F_\ell)|}{\abs{\slfrak_n(\F_q)}}
  \right)^2
  = q^{-2n^2(1 - 1/k) + 2}.
\]

We can now assemble the union bound. Given any pair $A, B$ that acts irreducibly on $\F_q^n$ but not absolutely irreducibly, its centralizer $D \cong \F_{\ell}$ with $k \geq 2$ is embedded into $\Mat_n(\F_q)$ as one of the copies counted above, and $A, B \in C_{\Mat_n(\F_q)}(D)$ holds by definition. Hence the event is contained in
\[
\bigcup_{\substack{k \geq 2 \\ k \mid n}} 
\bigcup_{\substack{\ell = q^k \\ \F_\ell \leq \Mat_n(\F_q)}}
\{ A, B \in C_{\Mat_n(\F_q)}(\F_\ell) \},
\]
and so its probability is at most
\[
  \sum_{\substack{k \mid n \\ k \geq 2}} 
  4 q^{n^2(1 - 1/k)}
  q^{-2n^2(1 - 1/k) + 2}
  = 4 q^2 \sum_{\substack{k \mid n \\ k \geq 2}} q^{-n^2(1 - 1/k)}
  \leq 4 n q^{2 - n^2/2}. \qedhere
\]
\end{proof}

Combining the two cases gives the main result of this section.

\begin{theorem} \label{thm:absolutely_irreducible}
Let $A, B \in \slfrak_n(\F_q)$ be uniformly random. Then the Lie algebra generated by $A, B$ acts absolutely irreducibly on $\F_q^n$ with probability $1 - O(n q^{1-n})$. 
\end{theorem}

\section{Largish block and irreducibility give everything}
\label{sec:largish-block-and-irreducibility-give-everything}

We have now reached the point where a random pair in $\slfrak_n(\F_q)$ generates, after a scalar extension to $\overline{\F}_q$, a Lie subalgebra of $\slfrak_n(\overline{\F}_q)$ containing an upper left block $\slfrak_m(\overline{\F}_q)$ with $m\to\infty$, with high probability and after a change of basis. Moreover, the Lie subalgebra acts absolutely irreducibly on the vector space $(\overline{\F}_q)^{n}$ with high probability. We now show that these two features already force the subalgebra to be the whole $\slfrak_n(\overline{\F}_q)$.

\subsection{Elementary lemmas}

In the proof, we will require the following facts.

\begin{lemma}\label{lem:isotypic-split}
Let $M=\bigoplus_{i=1}^{r}M_i$ be a finite dimensional module over an associative or Lie algebra, and suppose that $M_i$ and $M_j$ have no common composition factor for $i\neq j$. Then every submodule $N \leq M$ splits as $N=\bigoplus_{i=1}^{r} (N\cap M_i)$.
\end{lemma}
\begin{proof}
By induction it suffices to treat the case $r=2$. So let $M = M_1\oplus M_2$ with projections $\pi_i\colon M\to M_i$. Put $N'=(N\cap M_1)\oplus(N\cap M_2)\leq N$ and $Q=N/N'$. Note that $N/(N\cap M_2)$ surjects onto $Q$. On the other hand, since $\ker(\pi_1|_N)=N\cap M_2$, we have $N/(N\cap M_2)\cong\pi_1(N)\leq M_1$. Thus every composition factor of $Q$ is one of $M_1$. By symmetry every composition factor of $Q$ is also one of $M_2$. As $M_1$ and $M_2$ share none, it follows that $Q=0$ and $N=N'$.
\end{proof}

\begin{lemma}\label{lem:not-selfdual}
Let $M$ be a vector space over a field $\F$ with $\dim M \geq 3$. Then the natural
$\slfrak(M)$-module $M$ and its dual $M^{*}$ are not isomorphic.
\end{lemma}
\begin{proof}
An element $X \in \slfrak(M)$ acts on $M^*$ by $Xf = - f \circ X$ for $f \in M^*$. Therefore an isomorphism $\phi\colon M\to M^{*}$ of $\slfrak(M)$-modules is the same
as a nondegenerate bilinear form $\langle v,w\rangle = \phi(v)(w)$ on $M$ that is
$\slfrak(M)$-invariant in the sense
$\langle Xv,w\rangle=-\langle v,Xw\rangle$ for all $X\in\slfrak(M)$. Let us fix a basis of $M$ and write $\langle v,w\rangle=v^\top Gw$ with $G\in \GL(M)$. Invariance is then equivalent to $X^\top G=-GX$ for all $X\in\slfrak(M)$. Apply this with $X=E_{ij}$ with $i\neq j$. Comparing the $(j,l)$ entries of $E_{ji}G=-GE_{ij}$ gives $G_{il} = 0$ for all $l \neq j$. Since this holds for every $j\neq i$ and $\dim M\geq 3$, there are two distinct such indices $j$, and together they force the whole $i$-th row of $G$ to vanish. As $i$ was arbitrary, we obtain $G=0$. Thus $M \not\cong M^*$.
\end{proof}

\begin{lemma}\label{lem:schur-tensor}
Let $L$ be a Lie algebra over an algebraically closed field $\F$, let $U$ be an irreducible finite dimensional $L$-module. Let $N$ be a finite dimensional vector space over $\F$, viewed as a trivial $L$-module. Then every $L$-submodule of $U\otimes N$ has the form $U\otimes S$ for a unique subspace $S\leq N$. The same holds for $N \otimes U$.
\end{lemma}
\begin{proof}
Choose a basis $\mathcal N$ of $N$, so that we may identify $U\otimes N\cong \bigoplus_{n \in \mathcal N} U \otimes n$, which is a semisimple $U$-isotypic $L$-module. Consider the linear map
\[
N\to \hom_{L}(U, U\otimes N),\qquad
n\to (u\mapsto u\otimes n).
\]
We claim this is an isomorphism of vector spaces. It is clearly injective. Any $L$-morphism from $U$ to $U\otimes N$ is a tuple of endomorphisms of $U$, each of which is a scalar by irreducibility and Schur's lemma, and so the $L$-morphism is of the form $u\mapsto u\otimes n$ for some $n \in N$. Thus the above map is also surjective. Consequently, the simple submodules of $U\otimes N$ are exactly the subspaces $U\otimes n$ with $n\in N$, $n\neq0$.

Now let $W\leq U\otimes N$ be a submodule and set
\begin{equation} \label{eq:definition_of_D}
S=\{ n\in N \mid U\otimes n\leq W \},
\end{equation}
a subspace of $N$. Then $U\otimes S\leq W$. For the converse inclusion, note that $W$ is a submodule of the direct sum of simple modules $U \otimes N$, so it is itself a sum of its simple submodules. Each of these is some $U\otimes n\leq W$ with $n\in S$, and hence $W\leq U\otimes S$. Note also that the subspace $S$ can be recovered from $W$ by \cref{eq:definition_of_D}, giving uniqueness.
\end{proof}

\subsection{A block with irreducibility gives the full Lie algebra}

\begin{theorem}\label{thm:block-irred}
Let $V=M\oplus N$ be a finite dimensional vector space over an algebraically closed field $\F$ with $\dim M\geq 3$. Let $L\leq\slfrak(V)$ be a Lie subalgebra acting irreducibly on $V$. If $\slfrak(M)\leq L$, then $L=\slfrak(V)$.
\end{theorem}
\begin{proof}
Asssume that $N \neq 0$, otherwise there is nothing to prove. Consider $\End(V)$ as a module over $\slfrak(M)$ via the adjoint representation. Since $\slfrak(M)\leq L$, we can then view $L$ as an $\slfrak(M)$-submodule of $\End(V)$.

\smallskip

\noindent\emph{Step 1: the decomposition of $\End(V)$ with respect to $\slfrak(M)$.} 
Let us decompose
\begin{equation} \label{eq:decomposition_of_EndV_slA_module}
\End(V) = \End(M) \oplus \hom(M,N) \oplus \hom(N,M) \oplus \End(N),
\end{equation}
where $\hom(M,N)$ consists of the endomorphisms of $V$ with $N$ in the kernel and image in $N$, and similarly for the other blocks. These submodules can be understood in terms of the dual spaces $M^*$ and $N^*$, identified with the functionals on $V$ vanishing on $N$, respectively $M$. Thus $\slfrak(M)$ acts on $\End(M)=M\otimes M^*$ adjointly, on $\hom(M,N)\cong N\otimes M^*$ through the factor $M^*$, on $\hom(N,M)\cong M\otimes N^*$ through the factor $M$, and trivially on $\End(N)$. In this description, any $X \in \End(V)$ acts by
\begin{equation}\label{eq:rankone-bracket}
[X, v\otimes f]=(Xv)\otimes f - v\otimes(f\circ X)
\qquad
(v \in V, \ f \in V^*),
\end{equation}
where $v \otimes f$ is the rank $1$ endomorphism of $V$ given by $w \mapsto f(w)v$.

\smallskip

\noindent\emph{Step 2: the block decomposition of $L$.} Let $E = \End(M) \oplus \End(N)$, so that
\[
\End(V) = E \oplus \hom(M,N) \oplus \hom(N,M)
\]
as $\slfrak(M)$-modules. We claim these three summands pairwise share no composition factor. We have $\hom(M,N) \cong (M^*)^{\oplus \dim N}$ and $\hom(N,M) \cong M^{\oplus \dim N}$ with $M$ and $M^*$ irreducible. Every composition factor of $\hom(M,N)$ is thus isomorphic to $M^{*}$, and every composition factor of $\hom(N,M)$ is isomorphic to $M$. The block $\End(N)$ is a trivial $\slfrak(M)$-module, while $\End(M)=\glfrak(M)$ is the adjoint module with composition series $0\leq \F \mathrm{id}_M \cap \slfrak(M) \leq \slfrak(M) \leq \glfrak(M)$, whose factors are the trivial module and the simple module $\pslfrak(M)$ of dimension at least $(\dim M)^2 - 2$. Thus neither $M$ nor $M^{*}$ is a composition factor of $E$. Moreover, by Lemma \ref{lem:not-selfdual} we have $M\not\cong M^{*}$. Therefore $E$, $\hom(M,N)$ and $\hom(N,M)$ pairwise share no composition factor. Lemma \ref{lem:isotypic-split}
applied to the submodule $L\leq\End(V)$ then gives
\begin{equation} \label{eq:decomposition_of_L}
L
=
\bigl(L\cap E \bigr) 
\oplus 
\bigl(L\cap\hom(M,N)\bigr) 
\oplus 
\bigl(L\cap\hom(N,M)\bigr).
\end{equation}

\medskip
\noindent\emph{Step 3: the off-diagonal pieces are structured.}
As an $\slfrak(M)$-module, we have $\hom(M,N)\cong N\otimes M^*$ with $\slfrak(M)$ acting through the factor $M^*$. The natural module $M$ is irreducible, and hence so is $M^*$. Now apply Lemma~\ref{lem:schur-tensor} with $U=M^*$ to the submodule $L\cap\hom(M,N)$. Note that $\F n \otimes M^*=\hom(M,\F n)$, hence by \cref{eq:definition_of_D} we obtain
\[
L\cap\hom(M,N)=\hom(M,S),
\qquad
S=\{n\in N \mid \hom(M,\F n)\leq L\} \leq N.
\]
Dually, by taking $U=M$ we obtain
\[
L\cap\hom(N,M)=M\otimes T,
\qquad
T=\{\nu\in N^* \mid M\otimes\nu\leq L\}\leq N^*.
\]
Moreover, we have $S\neq 0$, since otherwise $L\cap\hom(M,N)=0$, so by \cref{eq:decomposition_of_L} every element of $L$ preserves $M$, contradicting irreducibility. Similarly $T \neq 0$.

\smallskip

\noindent\emph{Step 4: the lower-triangular pieces are full columns.} 
Let $e\in L\cap E$, written as $e=e_M+e_N$ with $e_M\in\End(M)$, $e_N\in\End(N)$. For $n\in S$ and $\mu\in M^*$ we have $n\otimes\mu\in\hom(M,\F n)\leq L$, and \cref{eq:rankone-bracket} gives
\[
[e,n\otimes\mu]
=
e_N(n)\otimes\mu - n\otimes(\mu\circ e_M)
\in L\cap\hom(M,N) = \hom(M,S).
\]
The second summand lies in $\hom(M,\F n)\leq\hom(M,S)$ since $n \in S$. It follows that $e_N(n)\otimes\mu\in\hom(M,S)$ for every $\mu$, and hence $e_N(n)\in S$. Thus $e_N(S)\leq S$ for all $e\in L\cap E$. It now follows from \cref{eq:decomposition_of_L} that $M \oplus S$ is $L$-invariant: any element of $L\cap E$ preserves $M$ and $S$, any element of $L \cap \hom(M,S)$ sends $M$ into $S$ and kills $N$, and any element of $L\cap\hom(N,M)$ sends $N$ into $M$ and kills $M$. As $M\oplus S\neq 0$, irreducibility forces $M\oplus S=V$, meaning that $S=N$ and so $\hom(M,N)\leq L$.

\smallskip

\noindent\emph{Step 5: the upper-triangular pieces are full rows.} For $e\in L\cap E $ and $\nu\in T$, \cref{eq:rankone-bracket} applied to $m\otimes\nu\in M\otimes T\leq L$
gives
\[
[e,m\otimes\nu]
=
e_M(m)\otimes\nu - m\otimes(\nu\circ e_N)
\in L\cap\hom(N,M) = M\otimes T,
\]
and as the first term lies in $M\otimes T$ we get $\nu\circ e_N\in T$. Hence
$K=\bigcap_{\nu\in T}\ker\nu\leq N$ is $e_N$-invariant for every
$e\in L\cap E $. As before, it now follows from \cref{eq:decomposition_of_L} that $K$ is $L$-invariant: elements of $L\cap\hom(N,M)=M\otimes T$ kill $K$, elements of $L\cap E $ preserve it, and elements of $\hom(M,N)\leq L$ kill $N\geq K$. Since $K$ is a proper subspace of $V$, irreducibility gives $K=0$, meaning that $T=N^*$ and so $\hom(N,M)\leq L$.

\smallskip

\noindent\emph{Step 6: the diagonal block.} We now have $\slfrak(M),\hom(M,N),\hom(N,M)\leq L$. For $m\in M$, $\mu\in M^*$, $n\in N$, $\nu\in N^*$, both $m\otimes\nu\in\hom(N,M)$ and $n\otimes\mu\in\hom(M,N)$ lie in
$L$, and thus \cref{eq:rankone-bracket} yields
\[
[m\otimes\nu,n\otimes\mu]
=
\nu(n) (m\otimes\mu)-\mu(m)(n\otimes\nu)
\in L,
\]
with $m\otimes\mu\in\End(M)$ and $n\otimes\nu\in\End(N)$. Taking $\mu(m)=1$,
$\nu(n)=0$ gives every rank one traceless operator on $N$, so $\slfrak(N)\leq L$. Taking $\mu(m)=\nu(n)=1$ gives $w=m\otimes\mu-n\otimes\nu\in L$ with
$\tr_M w=1$, $\tr_N w=-1$. Thus $L$ contains the subspaces $\slfrak(M)$, $\slfrak(N)$,
$\hom(M,N)$, $\hom(N,M)$, $\F w$, and so we finally obtain $L=\slfrak(V)$.
\end{proof}

\subsection{The main theorem in odd characteristic}

We are now ready to collect the previous results to deduce the main theorem in odd characteristic. When $n$ tends to infinity, we can extract the following quantitative bound.

\begin{theorem} \label{thm:main-quantitative}
For every $\gamma > 0$ there is a constant $n_0 > 0$ such that the following holds. For all $n \geq n_0$ and all odd prime powers $q$, uniformly random $A, B \in \slfrak_n(\F_q)$ generate the Lie algebra $\slfrak_n(\F_q)$ with probability at least $1 - n^{-\gamma}$.
\end{theorem}
\begin{proof}
Let us fix $\gamma > 0$ and apply \Cref{largish_block} with $\gamma + 1$ together with \Cref{thm:absolutely_irreducible}. Thus there are constants $C, n_0 > 0$ such that for $n \geq n_0$ a uniformly random pair $A, B \in \slfrak_n(\F_q)$ has, with probability at least $1 - n^{-\gamma - 1} - O(n q^{1-n}) \geq 1 - n^{-\gamma}$, the following two properties:
\begin{enumerate}
  \item there is a $P \in \GL_n(\overline{\F}_q)$ and an $m > \lceil C \log n \rceil$ such that the Lie algebra $L = P^{-1} ( \overline{\F}_q \otimes \langle A, B \rangle ) P \leq \slfrak_n(\overline{\F}_q)$ contains the upper left block $\slfrak_m(\overline{\F}_q)$,
  \item $\overline{\F}_q \otimes \langle A, B \rangle$ acts irreducibly on the natural module.
\end{enumerate}
On this event, $L$ also acts irreducibly. Write $\overline{\F}_q^{n} = M \oplus N$ with $M$ the block of dimension $m \geq 3$. \Cref{thm:block-irred} then gives that $L = \slfrak_n(\overline{\F}_q)$. Hence $\dim_{\F_q} \langle A, B \rangle = \dim_{\overline{\F}_q} L = n^2 - 1$, and so $\langle A, B \rangle = \slfrak_n(\F_q)$.
\end{proof}

\begin{proof}[Proof of \Cref{thm:main}, odd characteristic]
Let $\abs{\slfrak_n(\F_q)} \to \infty$. If $n \to \infty$, then the preceding theorem gives generation with probability $1 - o(1)$. Otherwise $n$ stays bounded, and $\abs{\slfrak_n(\F_q)} \to \infty$ forces $q \to \infty$. As $(n,p) \neq (3,3), (4,2)$, the Lie algebra $\slfrak_n(\F_q)$ is $2$-generated \cite{cantor2024two}, so its generating pairs form a nonempty Zariski open subset of $\slfrak_n \times \slfrak_n$ over $\overline{\F}_q$ defined by equations of degree depending only on $n$, and by the argument of \cite[Proposition 4.5]{barbieri2025diameter} a random pair lies in it with probability $1 - o(1)$ as $q \to \infty$.
\end{proof}

\section{Even characteristic}
\label{sec:even-characteristic}

We now handle the even characteristic case. In the argument so far, the key point where characteristic enters is in exploiting the uniqueness of \emph{differences} $\lambda_i - \lambda_j$ of eigenvalues. This appears both in controlling the characteristic polynomial (\Cref{thm:good-medium-factor}) and in extracting elementary matrices from a random pair (\Cref{prop:extraction}). In characteristic $2$ the uniqueness of these differences degenerates, since $\lambda_i - \lambda_j = \lambda_j - \lambda_i$. In this section we explain the modifications needed to cover $p = 2$. The bulk of the argument is unchanged, since the semisimplification and Vandermonde isolation of \Cref{sec:extraction} (\Cref{dim_adNXY}, \Cref{lie_algebra_contains_B_delta}), the block law of \Cref{lem:block-law}, the absolute irreducibility of \Cref{sec:absolute-irreducibility}, and the structural result of \Cref{sec:largish-block-and-irreducibility-give-everything} never use the characteristic. The key modification is to replace the differences by \emph{sums}.

\subsection{Distinct sums}

In characteristic $2$, the ordered difference multiset $\Diff(f)$ has every element of multiplicity at least $2$, since $\lambda_i - \lambda_j = \lambda_j - \lambda_i$, so the condition that all differences be unique can never hold. We therefore work instead with the multiset of unordered sums
\[
  \Sum(f) = 
  \{  \lambda_i + \lambda_j 
  \mid
  \lambda_i, \lambda_j \in \Roots(f), \ \lambda_i \neq \lambda_j \}.
\]
We say that \emph{all sums are unique} if $\lambda_i + \lambda_j = \lambda_k + \lambda_l$ forces $\{i,j\} = \{k,l\}$. With this replacement, the results of \Cref{sec:differences-of-roots} hold without change with the same proofs (in characteristic $2$ the linearized polynomial $L(\Fr)=\Fr^i+\Fr^j+\Fr^k+\id$ vanishes identically only when $\{i,j\}=\{k, 0\}$). Consequently \Cref{sec:char-poly-of-random-matrices} continues to hold, and it now produces an irreducible factor $f$ of degree $> \lceil C \log n \rceil$ so that all elements of $\Sum(f)$ have multiplicity $1$ in $\Sum(\chi_A)$.

\subsection{Extracting a pair of elementary matrices}

\Cref{lie_algebra_contains_B_delta} likewise needs no change. The essential difference from odd characteristic, however, is that in characteristic $2$, it does not produce an elementary matrix from a unique sum. Instead, we have $\ad_X(E_{ij}) = (\lambda_i + \lambda_j) E_{ij}$, and uniqueness of sums in $\Sum(f)$ implies the distinct values $\lambda_i + \lambda_j$ separate only the \emph{unordered} pairs $\{i,j\}$. The Vandermonde argument therefore isolates symmetric pairs
\[
  S_{\{i,j\}} = Y_{ij} E_{ij} + Y_{ji} E_{ji}
\]
instead of a single elementary matrix.

\subsection{Producing an elementary matrix}

We now show how to break the symmetry and obtain a single elementary matrix from the symmetric pairs. We work on the $m \times m$ block coming from an irreducible factor of degree $m$ with unique sums. By \Cref{lem:block-law}, in a suitable basis the entries satisfy $Y_{ij} = \Fr^i(Z_{j-i})$, where the variables $Z_1, \dots, Z_{m-1} \in \F_{\ell}$, $\ell = q^m$, are independent and uniformly random.

\begin{lemma} \label{lem:char2-seed}
Let $m \geq 7$. Let $i,j,k \in \Z/m\Z$ be distinct, and suppose the six differences $\pm(j-i)$, $\pm(k-j)$, $\pm(i-k)$ are distinct in $\Z/m\Z$. Set
\[
  T_{ijk} = Y_{ij} Y_{jk} Y_{ki} + Y_{ik} Y_{kj} Y_{ji}.
\]
Then $\P(T_{ijk} = 0) \leq 2 \ell^{-1}$, and whenever $T_{ijk} \neq 0$, the Lie algebra generated by $S_{\{ i,j \}}, S_{\{ j,k \}}, S_{\{ i,k \}}$ contains the elementary matrices $E_{ik}$, $E_{ki}$.
\end{lemma}
\begin{proof}
The only nonzero brackets among the elementary matrices appearing as summands of $S_{\{ i, j \}}$ and $S_{\{ j, k \}}$ are $[E_{ij}, E_{jk}] = E_{ik}$ and $[E_{ji}, E_{kj}] = E_{ki}$, so
\[
  [S_{\{i,j\}}, S_{\{j,k\}}] = Y_{ij} Y_{jk} E_{ik} + Y_{ji} Y_{kj} E_{ki}.
\]
Together with $S_{\{i,k\}} = Y_{ik} E_{ik} + Y_{ki} E_{ki}$, these two elements span the subspace $\langle E_{ik}, E_{ki} \rangle$ if and only if
\[
  \det \begin{pmatrix} Y_{ik} & Y_{ki} \\ Y_{ij} Y_{jk} & Y_{ji} Y_{kj} \end{pmatrix}
  = Y_{ik} Y_{ji} Y_{kj} + Y_{ij} Y_{jk} Y_{ki}
  = T_{ijk} \neq 0.
\]
This proves the second part of the claim. For the probability bound, the six entries appearing in $T_{ijk}$ lie on the six distinct cyclic diagonals $\pm(j-i), \pm(k-j), \pm(i-k)$. Each entry is a Frobenius power of the corresponding $Z$ variable, so the six entries are independent and uniform in $\F_{\ell}$. Write $U = Y_{ij} Y_{jk} Y_{ki}$ and $V = Y_{ik} Y_{kj} Y_{ji}$. These are independent and identically distributed random variables, since each is a product of three independent uniform elements. Hence $\P(U = 0) = 1 - (1 - \ell^{-1})^3 \leq 3 \ell^{-1}$, while for $u \neq 0$ we have $\P(U = u) = (\ell - 1)^2 \ell^{-3} \leq \ell^{-1}$. Therefore
\[
  \P(T_{ijk} = 0) = \P(U = V) = \P(U = 0)^2 + \sum_{u \neq 0} \P(U = u)^2
  \leq \frac{9}{\ell^2} + \frac{\ell - 1}{\ell^2} \leq \frac{2}{\ell},
\]
using $\ell = q^m \geq 2^7$. \qedhere
\end{proof}

Indices as in the lemma clearly exist as long as $m \geq 7$, for example $(i,j,k) = (0,1,3)$ have differences $\pm 1, \pm 2, \pm 3$, and these are distinct in $\Z/m\Z$.

\subsection{Propagating the elementary matrix across the block}

It remains to show that once a single elementary matrix in the $m \times m$ block is available, we can propagate it using symmetric pairs to obtain the whole block. For this we will need that $\pm 1 \in D$, or equivalently $Z_1, Z_{m-1} \neq 0$, where as before
\[
  D = \{ d \in \Z/m\Z \mid Z_d \neq 0, \ d \neq 0 \},
\]
so that for $i \neq j$ we have $Y_{ij} \neq 0$ if and only if $j - i \in D$.

\begin{lemma} \label{lem:char2-propagate}
Suppose $\{ 1, -1 \} \subseteq D$. Then any elementary matrix $E_{i_0 j_0}$ \textup{(}$i_0 \neq j_0$\textup{)} together with the symmetric pairs $\{ S_{\{i,j\}} \mid i \neq j \}$ generates $\slfrak_m(\F_\ell)$.
\end{lemma}
\begin{proof}
For distinct $i, j, l$ we have
\[
  [E_{ij}, S_{\{j,l\}}] = Y_{jl} E_{il},
  \qquad
  [S_{\{l,i\}}, E_{ij}] = Y_{li} E_{lj}.
\]
These are nonzero elementary matrices precisely when $l - j, i - l \in D$. If we take $l = j \pm 1$ in the first bracket and use $\pm 1 \in D$, we therefore pass from $E_{ij}$ to $E_{i, j \pm 1}$ whenever the latter is off-diagonal. That is, within a fixed row we may shift the column by $\pm 1$, the only forbidden value being the row index $i$. Since $\Z/m\Z \setminus \{i\}$ is connected under $\pm 1$, every off-diagonal matrix in row $i$ can be reached. Symmetrically, taking $l = i \pm 1$ in the second bracket lets us shift the row by $\pm 1$ within a fixed column, the only forbidden value being the column index. Hence all off-diagonal elementary matrices lie in the generated algebra, and finally $[E_{ij}, E_{ji}] = E_{ii} - E_{jj}$ supplies the diagonal.
\end{proof}

Combining the two preceding lemmas gives the even characteristic analogue of \Cref{prop:block-generates}.

\begin{proposition} \label{prop:char2-block-generates}
Let $q$ be even and $m \geq 7$. The symmetric pairs $\{ S_{\{i,j\}} \mid i \neq j \}$ generate $\slfrak_m(\F_\ell)$ with probability at least $1 - 4 \ell^{-1}$.
\end{proposition}
\begin{proof}

By \Cref{lem:char2-seed}, we have $T_{0,1,3} \neq 0$ with probability at least $1 - 2 \ell^{-1}$, and then the generated algebra contains an elementary matrix. By \Cref{lem:char2-propagate}, this matrix together with the symmetric pairs generates $\slfrak_m(\F_\ell)$ as soon as $\pm 1 \in D$, that is $Z_1 \neq 0$ and $Z_{m-1} \neq 0$, which fails with probability $2 \ell^{-1}$. By a union bound, the symmetric pairs fail to generate $\slfrak_m(\F_\ell)$ with probability at most $4 \ell^{-1}$.
\end{proof}

\subsection{The main theorem in even characteristic}

We can now run the argument of \Cref{sec:obtaining-a-largish-block} in characteristic $2$. \Cref{power_diagonalizable} and \Cref{dim_adNXY} are characteristic free. Repeating the proof of \Cref{largish_block} with \Cref{thm:good-medium-factor} read in terms of sums and \Cref{prop:block-generates} replaced by \Cref{prop:char2-block-generates}, we obtain the following.

\begin{theorem} \label{thm:char2-largish-block}
For every $\gamma > 0$ there are constants $C, n_0 > 0$ such that the following holds. Let $q$ be even, let $n \geq n_0$, and let $A, B \in \slfrak_n(\F_q)$ be uniformly random. Then, with probability at least $1 - n^{-\gamma}$, there are $P \in \GL_n(\overline{\F}_q)$ and $m > \lceil C \log n \rceil$ such that the Lie algebra generated by $P^{-1} A P, P^{-1} B P$ contains the upper left block $\slfrak_m(\overline{\F}_q)$.
\end{theorem}

The latter combined with absolute irreducibility exactly as in odd characteristic yields the even characteristic part of \Cref{thm:main}.

\bibliographystyle{alpha}
\bibliography{biblio}

\end{document}